\documentclass{amsart}
\usepackage{amssymb}
\usepackage{bbm} 
\usepackage{bm} 
\usepackage[nobysame]{amsrefs}
\usepackage{todonotes}
\usepackage{mathpazo}
\usepackage{enumerate}
\usepackage{rotating}
\usepackage{subcaption}
\usepackage{ifthen}
\usepackage{tikz}
	\usetikzlibrary{shapes,calc,decorations.pathmorphing}
	\pgfdeclarelayer{background}
	\pgfdeclarelayer{middle}
	\pgfsetlayers{background,middle,main}
	\tikzstyle{edge}=[line width=.75pt]
	\tikzstyle{fnode}=[fill=black,draw=black,circle,scale=\s]
	\tikzstyle{pathnode}=[inner sep=.9pt]

\usepackage[colorlinks=true,
	urlcolor=blue!20!black,
	citecolor=green!50!black]{hyperref}
\usepackage[top=3cm,bottom=3cm,left=2.5cm,right=2.5cm]{geometry}
\newtheorem{theorem}{Theorem}[section]
\newtheorem{proposition}[theorem]{Proposition}

\newtheorem{lemma}[theorem]{Lemma}
\newtheorem{corollary}[theorem]{Corollary}
\theoremstyle{remark}

\newtheorem{example}[theorem]{Example}
\newtheorem{remark}[theorem]{Remark}
\newtheorem{question}[theorem]{Question}

\newcommand{\defn}[1]{{\color{green!50!black}\emph{#1}}}
\newcommand{\defs}{\stackrel{\mathsf{def}}{=}}
\newcommand{\ie}{\text{i.e.}\;}

\newcommand{\SC}{\Delta}
\newcommand{\sphere}{\mathcal{S}}
\newcommand{\ball}{\mathcal{B}}

\renewcommand{\dim}{\mathsf{dim}}

\newcommand{\link}[2]{\mathsf{lk}_{#2}{#1}}
\newcommand{\interior}{\mathsf{int}}

\newcommand{\Z}{\mathbb{Z}}
\newcommand{\N}{\mathbb{N}}
\newcommand{\x}{\mathbf{x}}
\newcommand{\bolda}{\mathbf{a}}
\newcommand{\boldb}{\mathbf{b}}
\newcommand{\boldc}{\mathbf{c}}
\newcommand{\fieldk}{\mathbbm{k}} 
\newcommand{\SRring}{\fieldk[\SC]} 
\newcommand{\boldlambda}{\bm{\lambda}}
\newcommand{\boldomega}{\bm{\omega}}

\title[Dehn--Sommerville Relations]{Revisiting Generalizations of the Dehn--Sommerville Relations}
\author{Cesar Ceballos}
\address{CC: TU Graz, Institut f\"ur Geometrie, Kopernikusgasse 24, 8010 Graz, Austria.}
\email{cesar.ceballos@tugraz.at}
\author{Henri M{\"u}hle}
\address{HM: TU Dresden, Institut f{\"u}r Algebra, Zellescher Weg 12--14, 01069 Dresden, Germany.}
\email{henri.muehle@tu-dresden.de}
\thanks{Cesar Ceballos was supported by the Austrian Science Foundation FWF, grant P 33278.  Henri M{\"u}hle has received funding from the European Research Council (Grant Agreement no. 681988, CSP-Infinity.)}


\begin{document}

\allowdisplaybreaks

\begin{abstract}
We revisit several known versions of the Dehn--Sommerville relations in the context of:
\begin{itemize}
\item homology manifolds;
\item semi-Eulerian complexes;
\item general simplicial complexes;
\item balanced semi-Eulerian complexes; and
\item general completely balanced complexes.
\end{itemize}
In addition, we present Dehn--Sommerville relations for
\begin{itemize}
\item reciprocal complexes; and
\item general balanced simplicial complexes;
\end{itemize}
which slightly generalize some of the previous results.

Our proofs are uniform, and are based on two simple evaluations of the $\tilde h$-polynomial: 
one that recovers the $\tilde f$-polynomial, and one that counts faces according to certain multiplicities.    
\end{abstract}

\maketitle

\section{Introduction}
	\label{sec:introduction}
The Dehn--Sommerville relations are historically one of the key stones in the study of $f$-vectors of polytopes, simplicial complexes and manifolds in general. They constitute a system of linear relations that the $f$-numbers satisfy, which generalize the well known Euler-Poincar\'e formula for simplicial polytopes (or spheres):
\[
f_0-f_1+f_2-f_3+\dots + (-1)^{d-1}f_{d-1} = (-1)^{d-1} + 1
\] 	
In their simpler form, for instance for simplicial polytopes or simplicial spheres, the Dehn--Sommerville relations can be stated as:
\[
 f_{k-1} = \sum_{i=k}^d (-1)^{d-i} {i \choose k} f_{i-1},
\]
where $f_j$ counts the number of $j$-dimensional faces of the complex. 
Equivalently,
\[
h_{d-i} = h_i
\]
where $h_0,h_1,\dots,h_d$ are the so called $h$-numbers.

In his seminal work from 1964, Klee proved a version of the Dehn--Sommerville relations for a more general class of simplicial complexes called semi-Eulerian complexes~\cite{klee_DehnSommervilleRelations_1964}.
Since then, several other versions have been rediscovered by many authors over and over again.  
The following list includes some generalizations of the Dehn--Sommerville relations for:
\begin{itemize}
\item semi-Eulerian complexes; Klee~\cite{klee_DehnSommervilleRelations_1964}.
\item homology manifolds; Macdonald~\cite{Macdonald_DehnSommerville_1971} (see Remark~\ref{rem_Macdonald} and Appendix~\ref{appendix_Macdonald}), Gr\"abe~\cite{grabe_DehnSommerville_1987}, Chen-Yan~\cite{ChenYan_DehnSomerville_1997}, Novik--Swartz~\cite{novik_applications_2009}.  
\item completely balanced spheres (or more generally, completely balanced Eulerian complexes); Bayer--Billera~\cite{bayer_generalized_1985}.
\item balanced semi-Eulerian complexes; Swartz~\cite{swartz_face_enumeration_2009}.
\item general simplicial complexes; Sawaske--Xue~\cite{sawaske_nonEulerianDehnSommerville_2021}.
\item general completely balanced complexes; Sawaske--Xue~\cite{sawaske_nonEulerianDehnSommerville_2021}.
\end{itemize}

In this paper, we present uniform proofs for all these results.
Our key ingredients are two elementary evaluations of the $\tilde h$-polynomial $\tilde h(x)= \sum_{i=0}^dh_ix^i$, which are stated in Theorem~\ref{thm_fh_tilde} and Theorem~\ref{thm_fh_reciprocity}. The first recovers the $\tilde f$-polynomial 
$\tilde f(x)=\sum_{i=0}^df_{i-1}x^i$,
while the second counts faces of the complex with certain multiplicities that depend on the reduced Euler characteristic of their links. 
From these two theorems one deduces a single polynomial relation; 
comparing the coefficients on both sides of this relation determines the Dehn--Sommerville relations.   

This allows us to extend the previous results for (1) a new family of simplicial complexes that we call reciprocal complexes, and for (2) general balanced simplicial complexes, extending results by Swartz~\cite{swartz_face_enumeration_2009} and Sawaske-Xue~\cite{sawaske_nonEulerianDehnSommerville_2021}. 

Finally, we relate Theorem~\ref{thm_fh_reciprocity}, which we call ``the $f=h$ reciprocity'', with a reciprocity result of Stanley about the Hilbert series of the Stanley--Reisner ring (Appendix~\ref{app_StanleyReisner}). This last result is a common tool in several of the existing proofs of the Dehn--Sommerville relations. We highlight that our proofs do no require any background about Stanley--Reisner rings.  

We would also like to mention that there are further generalizations of the Dehn--Sommerville relations which are not discussed in this manuscript. For instance, the Dehn--Sommerville relations for 
Eulerian posets by Bayer and Billera~\cite{bayer_generalized_1985}, 
completely balanced semi-Eulerian posets by Stanley~\cite{stanley_some_1982}, or
relative simplicial complexes in the weakly Eulerian case by Adiprasito and Sanyal~\cite{adiprasito_relative_2016}.  
The reader is invited to use the techniques presented in this paper to prove these results as well.

\section{The $f{=}h$ reciprocity for simplicial complexes}
	\label{sec_fh_reciprocity_DehnSommerville}

A \defn{simplicial complex} $\SC$ is a collection of sets (called faces) that is closed under taking subsets. The union of all these sets is called the \defn{vertex set} of the complex. The \defn{dimension}~$\dim(F)$ of $F\in \SC$ is $|F|-1$, and the \defn{dimension} $\dim(\SC)$ is the largest dimension of a face $F$ in~$\SC$. The \defn{link} of a face $F\in \SC$ is the set $\link{F}{\SC}$ of faces $G\in\SC$ such that $G \cap F =\emptyset$ and~$G\cup F\in \SC$.

From now on we consider a finite simplicial complex $\SC$ of dimension $d-1$ whose vertex set is $\{1,2,\dots , n\}$.
The \defn{$f$-vector} of $\SC$ is defined as 
\begin{displaymath}
	f(\SC)\defs(f_{-1},f_0,\dots, f_{d-1}),
\end{displaymath}
where $f_i$ is equal to the number of $i$-dimensional faces in~$\SC$ for $0\leq i \leq d-1$, and $f_{-1}\defs 1$ represents the empty face. 
The reduced Euler characteristic and Euler characteristic of $\SC$ are defined respectively by:  
\begin{align*}
	\widetilde \chi (\SC) \defs \sum_{i=0}^d (-1)^{i-1} f_{i-1}, &&
	\chi (\SC) \defs \sum_{i=1}^d (-1)^{i-1} f_{i-1}.
\end{align*}
In particular, $\widetilde \chi (\SC)=\chi (\SC)-1$, where the term $-1$ accounts for the empty face.
The \defn{$h$-vector} of $\SC$ is the vector 
\begin{align*}
	h(\SC)\defs(h_{0},h_1,\dots, h_{d}),
\end{align*}
which is defined in terms of the $f$-vector via the relation 
\begin{equation}\label{relation_fh_vectors}
	\sum_{i=0}^{d} h_i x^i
	=
	\sum_{i=0}^{d} f_{i-1} x^i(1-x)^{d-i}.  
\end{equation}

The study of $h$-vectors has historically proven to be a powerful tool in the study of $f$-vectors of polytopes, simplicial complexes and manifolds in general. Important results about their face numbers (e.g. the Dehn--Sommerville relations) can be elegantly formulated and/or proven using $h$-vectors instead.      

We define the \defn{$\tilde f$-polynomial} and the \defn{$\tilde h$-polynomial} of $\SC$ as:
\begin{align}\label{ftilde_htilde_polynomial}
	\tilde f(x) \defs \sum_{i=0}^d f_{i-1}x^{i}, &&
	\tilde h(x) \defs \sum_{i=0}^d h_{i}x^{i}.
\end{align}

The reason why we use a tilde in our notation is because there are commonly used notions of $f$-polynomials and $h$-polynomials in the literature, which are defined by
\begin{align}
	f(x) \defs \sum_{i=0}^d f_{i-1}x^{d-i}, &&
	h(x) \defs \sum_{i=0}^d h_{i}x^{d-i}.
\end{align}
The relation~\eqref{relation_fh_vectors} between the $f$- and the $h$-vector is equivalent to the simpler formulation 
\begin{align}\label{eq_fh}
	f(x)=h(x+1).
\end{align}
 Our definitions of the $\tilde f$- and $\tilde h$-polynomial reverse the coefficients of the $f$- and $h$-polynomial.
These choices are more convenient for our purposes. 

In particular, the various generalizations of the Dehn--Sommerville relations presented in this paper are simple consequences of the following two basic results. 
The first is an evaluation of the $\tilde h$-polynomial that recovers the $\tilde f$-polynomial, while the second is another evaluation that counts faces with certain multiplicities. 

\begin{theorem}\label{thm_fh_tilde}
Let $\SC$ be an abstract simplicial complex of dimension $d-1$
with $\tilde f$-polynomial  $\tilde f(x)$ and $\tilde h$-polynomial  $\tilde h(x)$. 
The following relation holds:
\begin{align}\label{eq_fh_tilde}
	 (x+1)^d \tilde h\left(\frac{x}{x+1}\right) = \tilde f(x).
\end{align}
\end{theorem}
\begin{proof}
Replacing $x$ by $\frac{x}{x+1}$ in~\eqref{relation_fh_vectors}, and multiplying the result by $(x+1)^d$ yields
\begin{align*}
(x+1)^d \tilde h\left(\frac{x}{x+1}\right) 
= (x+1)^d \sum_{i=0}^{d} f_{i-1} \left(\frac{x}{x+1}\right)^i\left( \frac{1}{x+1}\right)^{d-i}
= \sum_{i=0}^{d} f_{i-1} x^i
= \tilde f(x),
\end{align*}
as desired.
\end{proof}

\begin{theorem}[The $f{=}h$ reciprocity for simplicial complexes]
\label{thm_fh_reciprocity}
Let $\SC$ be an abstract simplicial complex of dimension $d-1$ with $\tilde h$-polynomial  $\tilde h(x)$. Then, 
\begin{equation}\label{eq_fh_reciprocity}
 x^d \tilde h\left(\frac{x+1}{x}\right) =
  \sum_{F\in \SC} m_F x^{|F|},  
\end{equation}
where 
\begin{equation}\label{eq_multiplicity}
	m_F \defs \sum_{F\subseteq G \in \SC} (-1)^{d-|G|}  = (-1)^{d-1-|F|}\widetilde \chi\bigl(\link{F}{\SC}\bigr).
\end{equation}
That is, the evaluation $x^d \tilde h(\frac{x+1}{x})$ enumerates the faces of the complex with multiplicity~$m_F$, which is given by a signed reduced Euler characteristic of their links. In particular, the multiplicity of the empty face is
\begin{equation}\label{eq_multiplicity_empty}
	m_\emptyset = (-1)^{d-1}\widetilde \chi\bigl(\SC\bigr).
\end{equation}
 
\end{theorem}

\begin{proof}
From a direct substitution in~\eqref{relation_fh_vectors} we get:
\begin{equation*}\label{eq_hf_substitution}
 x^d \tilde h\left(\frac{x+1}{x}\right) = 
 \sum_{i=0}^d f_{i-1} (x+1)^i (-1)^{d-i}.
\end{equation*}
Let $\mathbf{x}_F=\prod_{i\in F} x_i$, and define the auxiliary multivariate function 
\begin{align*}
\mathbf{\tilde f}(x_1,\dots,x_n) &\defs \sum_{G\in \SC} \left( \sum_{F\subseteq G} \mathbf{x}_F \right) (-1)^{d-|G|}\\
&= \sum_{F\in \SC} \mathbf{x}_F \left( \sum_{F\subseteq G\in \SC} (-1)^{d-|G|} \right)
\end{align*}
Equation~\eqref{eq_hf_substitution} can then be expressed as
\begin{align*}
x^d \tilde h\left(\frac{x+1}{x}\right) = \mathbf{\tilde f}(x,\dots,x) 
= \sum_{F\in \SC} m_F x^{|F|}, 
\end{align*}
for $m_F = \sum_{F\subseteq G \in \SC} (-1)^{d-|G|}$ as desired. 
Moreover, $$d-|G|= \bigl(|G|-|F|-1\bigr)+ \bigl(d-1-|F|\bigr)+2\bigl(|F|-|G|+1\bigr)$$
and so, we have:
\begin{align*}
m_F &=  \sum_{F\subseteq G \in \SC} (-1)^{d-|G|}\\
& = \sum_{F\subseteq G \in \SC} (-1)^{|G|-|F|-1}  (-1)^{d-1-|F|}  \\
& = (-1)^{d-1-|F|} \widetilde \chi\bigl(\link{F}{\SC}\bigr).
\end{align*}
This proves~\eqref{eq_multiplicity}. 

Finally, for $F=\emptyset$ we have $\link{\emptyset}{\SC}=\SC$ and $m_\emptyset = (-1)^{d-1}\widetilde \chi\bigl(\SC\bigr)$. 
This proves~\eqref{eq_multiplicity_empty}. 
\end{proof}

Examples of Theorem~\ref{thm_fh_reciprocity} are presented in Section~\ref{sec_examples} and Figure~\ref{fig_examples_fh_reciprocity}.

\begin{remark}
So far, we were not able to find an explicit reference for Theorem~\ref{thm_fh_reciprocity} in the literature, at least in this basic form.
However, we will show in Appendix~\ref{app_StanleyReisner} that Theorem~\ref{thm_fh_reciprocity} (and its multi-variate generalization in Theorem~\ref{thm_fh_reciprocity_multi}) is equivalent to a specialization of a reciprocity result by Stanley~\cite[Theorem~7.1]{stanley_combinatorics_algebra_book_1996} about the Hilbert series of the Stanley--Reisner ring. 
\end{remark}

\section{Homology manifolds}\label{sec_homology}
Theorem~\ref{thm_fh_reciprocity} has nice implications in the case of homology manifolds. 
As we will see shortly, the multiplicity~$m_F$ of a non-empty face in a homology manifold is equal to zero when $F$ is a boundary face and to one otherwise. 
The face enumeration of homology manifolds is studied comprehensively by Swartz in~\cite{swartz_face_enumeration_2009}, and the definitions presented here follow his conventions, see~\cite[Section~2]{swartz_face_enumeration_2009}. 

Fix a field $\fieldk$. A simplicial complex $\SC$ is called a \defn{homology manifold} if for every non-empty face $F\in\SC$, 
the reduced simplicial homology $\widetilde H_i\bigl(\link{F}{\SC};\fieldk\bigr)$ vanishes if $i<d-1-|F|$ and is isomorphic to $\fieldk$ or $0$ if $i=d-1-|F|$.
In other words, the $\fieldk$-homology 
$H_\star\bigl(\link{F}{\SC};\fieldk\bigr)$ is isomorphic to either the $\fieldk$-homology of a sphere $\sphere^{d-1-|F|}$ or a ball $\ball^{d-1-|F|}$ (the exponent here denotes the dimension). 
A non-empty face $F\in \SC$ is called a \defn{boundary face} if 
$H_{d-1-|F|}\bigl(\link{F}{\SC};\fieldk \bigr)=0$, that is, if its link has the homology of a ball.
The \defn{boundary} $\partial \SC$ is defined as the sub-complex consisting of the empty face and all non-empty boundary faces.  
 The faces that are not contained in the boundary are called \defn{interior}, that is the non-empty faces whose links have the homology of a sphere.

\begin{proposition}\label{prop_hom_manifold_m_F}
Let $\SC$ be a homology manifold. For a non-empty face $F\in \SC$,
\[
m_F=
\begin{cases}
0 & \text{if } F\in \partial \SC, \\
1 & \text{otherwise.}
\end{cases}
\]
\end{proposition}
\begin{proof}
Recall that the reduced Euler characteristic of a ball (regardless of its dimension) is equal to zero, and that the reduced Euler characteristic of an $\ell$-dimensional sphere is equal to $(-1)^\ell$. Since the reduced Euler characteristic is a homology invariant, we deduce that for a non-empty face $F$, 
\[
m_F = (-1)^{d-1-|F|}\widetilde \chi\bigl(\link{F}{\SC}\bigr)
\]
is equal to zero if $\link{F}{\SC}$ has the homology of a ball, and is equal to one if it has the homology of a sphere. These are precisely the definitions of boundary and interior faces, respectively. 
\end{proof}

We denote by $f_{i}^{\interior}$ the number of $i$-dimensional interior faces of $\SC$ for $i\in\{0,1,\ldots,d\}$, and define the \defn{$\tilde f^\interior$-polynomial} by
\begin{align}\label{eq_tilde_f_interior_polynomial}
	\tilde f^\interior (x) \defs \sum_{i=1}^d f_{i-1}^{\interior}x^i.
\end{align} 
Note that the sum starts at $i=1$ and does not include the empty face.

\begin{corollary}
\label{cor_fh_homology_manifolds}
Let $\SC$ be a homology manifold of dimension $d-1$. Then
\begin{equation}\label{eq_interiorfacesreciprocity}
  x^d \tilde h\left(\frac{x+1}{x}\right) =
\tilde f^\interior (x) + m_\emptyset.
\end{equation}
\end{corollary}
 
 \begin{proof}
 Applying Theorem~\ref{thm_fh_reciprocity}, we see that the terms in the sum of Equation~\eqref{eq_fh_reciprocity}
 that correspond to (non-empty) boundary faces vanish ($m_F=0$), while the ones corresponding to interior faces are counted with multiplicity one ($m_F=1$). Therefore,
\[
x^d \tilde h\left(\frac{x+1}{x}\right) 
=   \sum_{\emptyset \neq F\in \SC} m_F x^{|F|}  + m_\emptyset
=  \sum_{i=1}^d f_{i-1}^\interior x^i + m_\emptyset.
\]
\end{proof}

\begin{theorem}[Polynomial Dehn--Sommerville relation for homology manifolds]
\label{thm_polynomialDehnSommerville_homology}
Let $\SC$ be a homology manifold of dimension $d-1$. Then, 
\begin{align}
(-1)^d \tilde f(x) = \tilde f^\interior (-x-1) + m_\emptyset. \label{eq_f_fint_one}
\end{align}
\end{theorem}
\begin{proof}
Replacing $x$ by $-x-1$ in Equation~\eqref{eq_interiorfacesreciprocity} yields
\begin{equation*}
  (-1)^d (x+1)^d \tilde h\left(\frac{x}{x+1}\right) =
\tilde f^\interior (-x-1) + m_\emptyset.
\end{equation*}
The left hand side is equal to $(-1)^d \tilde f(x)$ by Theorem~\ref{thm_fh_tilde}.
\end{proof}

\begin{remark}\label{rem_Macdonald}
In 1971, Macdonald proved a polynomial relation satisfied for finite simplicial complexes whose underlying topological space is a manifold~\cite[Theorem~2.1]{Macdonald_DehnSommerville_1971}. In Appendix~\ref{appendix_Macdonald}, we will show that Equation~\eqref{eq_f_fint_one} implies Macdonald's relation, and explain why Macdonald's relation is less general.
As far as we know, Macdonald's polynomial relation is the first appearance of a weaker instance\footnote{Not all relations in~\eqref{eq_gen_DehnSommerville_one} follow from Macdonald's polynomial relation; see Appendix~\ref{appendix_Macdonald} for details.} of Dehn--Sommerville relations for manifolds. 
As we will see in Corollary~\ref{cor_DehnSommerville_homology}, the polynomial relation~\eqref{eq_f_fint_one} implies a system of linear relations satisfied by the face and interior face numbers of the complex. Such linear relations were found independently by Gr\"abe in~1987~\cite[Theorem~2.1]{grabe_DehnSommerville_1987} and by Chen and Yan in~1997~\cite[Lemma~8]{ChenYan_DehnSomerville_1997} (both in the form of Equation~\eqref{eq_gen_DehnSommerville_version_two} further below). An $h$-version of the relations was proven by Novik and Swartz in~2009~\cite[Theorem~3.1]{novik_applications_2009}.

\end{remark}

\begin{corollary}[{\cites{Macdonald_DehnSommerville_1971,grabe_DehnSommerville_1987,ChenYan_DehnSomerville_1997,novik_applications_2009}} Dehn--Sommerville relations for homology manifolds, $f$-version]\label{cor_DehnSommerville_homology}
	Let $\SC$ be a homology manifold of dimension $d-1$. The following relations hold:
	\begin{align}
		\text{for } k\geq 1, && f_{k-1} &= \sum_{i=k}^d (-1)^{d-i} {i \choose k} f_{i-1}^\interior. \label{eq_gen_DehnSommerville_one}\\
		\text{for } k= 0, && f_{-1} &= \sum_{i=1}^d (-1)^{d-i}  f_{i-1}^\interior   +   (-1)^d m_\emptyset. \label{eq_gen_DehnSommerville_two} 
	\end{align} 
	The second relation is equivalent to the perhaps more elegant formulation:
	\begin{align}
		\chi(\SC) = (-1)^{d-1} \chi (\SC^\interior), \label{eq_gen_DehnSommerville_two_prime}
	\end{align}
	where $ \chi (\SC^\interior) \defs \sum_{i=1}^d (-1)^{i-1}f_{i-1}^\interior$.
\end{corollary}

\begin{proof}
This result is essentially a reformulation of the polynomial Dehn--Sommerville relation in Equation~\eqref{eq_f_fint_one}, which states that 
\[
(-1)^d \tilde f(x) = \tilde f^\interior \bigl(-(x+1)\bigr) + m_\emptyset.
\]
In order to prove the corollary, we just need to compare the coefficients of the polynomials on both sides of this 
equation.
Since $m_\emptyset=(-1)^{d-1}\widetilde \chi(\SC)$ is a constant number, we split the analysis into two cases: 
the coefficient of $x^k$ for $k\geq 1$, and the coefficient of $x^0$ (constant term, $k=0$).
 
For $k\geq 1$, we have
\begin{align*}
(-1)^d f_{k-1} &= \sum_{i=k}^d (-1)^i {i \choose k} f_{i-1}^\interior.
\end{align*}
Equation~\eqref{eq_gen_DehnSommerville_one} follows by multiplying by $(-1)^d$ and using $(-1)^{d+i}=(-1)^{d-i}$. 
 
For $k=0$, we proceed similarly using the extra constant term $m_\emptyset$:
\begin{align*}
(-1)^d f_{-1} &= \sum_{i=1}^d (-1)^i f_{i-1}^\interior + m_\emptyset.
\end{align*}
Equation~\eqref{eq_gen_DehnSommerville_two} follows by multiplying again by $(-1)^d$ and using $(-1)^{d+i}=(-1)^{d-i}$.

Finally, we need to argue that Equation~\eqref{eq_gen_DehnSommerville_two} is equivalent to Equation~\eqref{eq_gen_DehnSommerville_two_prime}.
For this we just need to simplify Equation~\eqref{eq_gen_DehnSommerville_two}. We obtain
\begin{align*}
1=  f_{-1} = (-1)^{d-1}\sum_{i=1}^d (-1)^{i-1}  f_{i-1}^\interior +   (-1)^d m_\emptyset =  (-1)^{d-1} \chi(\SC^\interior)+   (-1)^d m_\emptyset .
\end{align*}
By~\eqref{eq_multiplicity_empty}, we have that  
$1-(-1)^d m_\emptyset = 1+ \widetilde \chi(\SC)= \chi(\SC)$. 
Thus,
\begin{align*}
 \chi(\SC) =  (-1)^{d-1}\chi(\SC^\interior)
\end{align*}
as desired.
\end{proof}

\begin{example}\label{example_cylinder}
Consider the simplicial complex $\SC$ of dimension $d-1=2$ illustrated in Figure~\ref{fig_cylinder}. 
It is obtained by subdividing a rectangle into four triangles and identifying the top and bottom edges of the rectangle as shown. The result is a triangulated cylinder. This is a triangulated manifold with boundary, and in particular a homology manifold.

This triangulated manifold consists of 4 vertices,  8 edges and 4 triangles. The $f$-vector is then $(f_{-1},f_0,f_1,f_2)=(1,4,8,4)$. 
There are no interior vertices, only 4 of the edges are interior (the diagonal edges, and the top and bottom edge), and the 4 triangles are interior. Thus, the $f^\interior$-vector is $(f_0^\interior,f_1^\interior,f_2^\interior)=(0,4,4)$.

The Dehn--Sommerville relations in Corollary~\ref{cor_DehnSommerville_homology} give the following linear relations for $k\geq 1$:
\begin{align*}
 \text{for } k= 1, && f_{0} &=  f_0^\interior -2 f_1^\interior+3 f_2^\interior, \\
 \text{for } k= 2, && f_{1} &=  -f_1^\interior+3f_2^\interior, \\
 \text{for } k= 3, && f_{2} &= f_2^\interior, 
\end{align*}
which can be easily verified by plugging in the numbers.
For $k=0$, we obtain $\chi(\SC) = (-1)^{2} \chi (\SC^\interior)$, which is verified by 
\begin{align*}
f_0-f_1+f_2&=f_0^\interior-f_1^\interior+f_2^\interior \\
4-8+4 &= 0-4+4.
\end{align*}
Note that this last relation is just a linear combination of the first three. 
\end{example}

\begin{remark}
In general, Equation~\eqref{eq_gen_DehnSommerville_two_prime} (or equivalently, Equation~\eqref{eq_gen_DehnSommerville_two}) is a linear combination of the equations in~\eqref{eq_gen_DehnSommerville_one}. Thus,~\eqref{eq_gen_DehnSommerville_one} alone may be regarded as the Dehn--Sommerville relations.
\end{remark}
  
\begin{figure}
\includegraphics[scale=1,page=17]{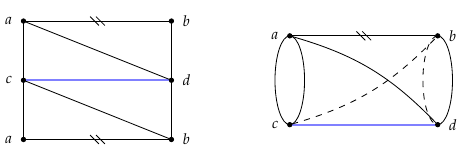}
\caption{Example of a homology manifold: a triangulated cylinder with boundary. The bottom and top edges of the rectangle on the left are identified giving rise to the cylinder on the right.}
\label{fig_cylinder}
\end{figure}

\section{Reciprocal complexes and semi-Eulerian complexes}
Motivated by homology manifolds, we introduce a more general family of simplicial complexes which we call reciprocal complexes.
A simplicial complex~$\SC$ is called \defn{reciprocal} if $m_F\in \{0,1\}$ for every non-empty face $F\in \SC$. 
A non-empty face $F\in \SC$ is called a \defn{boundary face} if $m_F=0$, and it is called \defn{interior} if~$m_F=1$. The empty face is considered to be a boundary face by definition. We say that $\SC$ is a complex \defn{without boundary} if the only boundary face is the empty face.
As above, we denote by $f_{i}^{\interior}$ the number of $i$-dimensional interior faces of $\SC$ for $i\in\{0,1,\ldots,d-1\}$, and by $\tilde f^\interior(x)$ its {$\tilde f^\interior$-polynomial}.

One interesting sub-class of examples is the collection of reciprocal complexes without boundary. These simplicial complexes
are well studied in the literature and are called semi-Eulerian. They were introduced in the seminal work~\cite{klee_DehnSommervilleRelations_1964} by Klee, where he proves that semi-Eulerian complexes satisfy a general version of the Dehn--Sommerville relations that was previously known for boundary complexes of simplicial polytopes.

A simplicial complex $\SC$ is called \defn{semi-Eulerian} if 
every non-empty face $F\in \SC$ has multiplicity $m_{F}=1$, 
or equivalently
 \begin{align}\label{eq_semiEulerian}
 \widetilde \chi\bigl(\link{F}{\SC}\bigr) = (-1)^{d-1-|F|}.
 \end{align}
 That is, the link of a non-empty face has the same reduced Euler characteristic as a sphere
 of dimension $d-1-|F|$.
 If in addition, $\widetilde \chi\bigl(\SC\bigr)= \widetilde \chi\bigl(\sphere^{d-1}\bigr) = (-1)^{d-1}$ where $\sphere^{d-1}$ denotes a $(d-1)$-sphere, then~$\SC$ is called an \defn{Eulerian complex} (this is equivalent to~$m_\emptyset=1$).\footnote{Semi-Eulerian complexes as stated here are called Eulerian complexes in~\cite{klee_DehnSommervilleRelations_1964}. We follow the convention used in~\cite{swartz_face_enumeration_2009}.}
For instance, 
every simplicial sphere is Eulerian.  
 
 The next
 lemma follows by definition.
\begin{lemma}\label{lem_semiEulerian_reciprocal}
A simplicial complex $\SC$ is semi-Eulerian if and only if $\SC$ is a reciprocal complex without boundary. 
In this case, $f_{i}^\interior = f_{i}$ for $0\leq i\leq d-1$.
\end{lemma}

Note that the proofs of the Dehn--Sommerville relations for homology manifolds in Section~\ref{sec_homology} only rely on the fact that $m_F\in \{0,1\}$ for every non-empty face $F\in \SC$. So, the same proofs apply for reciprocal complexes and semi-Eulerian complexes. 
Part~\eqref{item_part2_klee} of the following theorem, restricted to the case of semi-Eulerian complexes, is equivalent to~\cite[Theorem~3.2]{klee_DehnSommervilleRelations_1964}.  

\begin{theorem}[Dehn--Sommerville relations for reciprocal complexes, $f$-version]
Let $\SC$ be a reciprocal complex of dimension $d-1$ (for example a homology manifold or a semi-Eulerian complex). The following hold:
\begin{enumerate}
\item the polynomial Dehn--Sommerville relation~\eqref{eq_f_fint_one} in Theorem~\ref{thm_polynomialDehnSommerville_homology};
\item the Dehn--Sommerville relations~\eqref{eq_gen_DehnSommerville_one} and~\eqref{eq_gen_DehnSommerville_two}-\eqref{eq_gen_DehnSommerville_two_prime} in Corollary~\ref{cor_DehnSommerville_homology}. \label{item_part2_klee}
\end{enumerate}
\end{theorem}

Equation~\eqref{eq_gen_DehnSommerville_two_prime} has a nice implication for semi-Eulerian complexes. 
Let $\SC$ be an odd-dimensional semi-Eulerian complex (a reciprocal complex without boundary). Then, $f_i^\interior = f_i$ for $0\leq i\leq d-1$, and so  $\chi(\SC) =  \chi(\SC^\interior)$. Combining this with Equation~\eqref{eq_gen_DehnSommerville_two_prime}  
and the assumption that $d-1$ is odd yields $\chi(\SC)=-\chi(\SC)$.  Therefore, $\chi(\SC)=0$, 
which is equal to the Euler characteristic of an odd-dimensional sphere. From this we deduce that every odd-dimensional semi-Eulerian complex is Eulerian, a known result.

Equation~\eqref{eq_gen_DehnSommerville_one} expresses the face numbers of reciprocal complexes as linear combinations of the interior face numbers. Conversely, we can also express the interior face numbers as (the same) linear combinations of the face numbers. 

\begin{corollary}
\label{thm_generalizedDehnSommerville_reciprocal_version_two}
	Let $\SC$ be a reciprocal complex of dimension $d-1$ (for example a homology manifold or a semi-Eulerian complex). For $k\geq 1$, the following relations hold:
	\begin{align}
		 f_{k-1}^\interior &= \sum_{i=k}^d (-1)^{d-i} {i \choose k} f_{i-1}. \label{eq_gen_DehnSommerville_version_two}
	\end{align} 
\end{corollary}
\begin{proof}
The proof of this result is the same as the proof of~\eqref{eq_gen_DehnSommerville_one} in Corollary~\ref{cor_DehnSommerville_homology}. 
The only difference is that in this case we compare the coefficients of $x^k$ for $k\geq 1$ in the polynomial relation 
\begin{align}\label{eq_f_fint_two}
  (-1)^d \tilde f^\interior (x)  = \tilde f(-x-1) - (-1)^d m_\emptyset,
\end{align}
which is equivalent to the polynomial Dehn--Sommerville relation~\eqref{eq_f_fint_one} used in the proof of~\eqref{eq_gen_DehnSommerville_one}. 
\end{proof}

We invite the reader to check that these relations hold for our Example~\ref{example_cylinder}. 

\subsection{Two questions}
Reciprocal complexes are interesting objects on their own. They have similarities with the $SB$-lattices of P.~Hersh and K.~M{\'e}sz{\'a}ros in~\cite{hersh_sb_labelings_2017} (via their face posets), whose M\"obius function takes values from the set~$\{-1,0,1\}$. 

We highlight two (topological) questions about reciprocal complexes.   
The first question is about their relation to homology manifolds. As we pointed out already, homology manifolds are particular cases of reciprocal complexes, but one may wonder if the opposite statement holds.

\begin{question}
	Are there reciprocal complexes that are not homology manifolds?  
\end{question}

Isabella Novik~\cite{novik_personalCommunication_2021} pointed out to us an affirmative answer to this question: 

\begin{quotation}
``
There are plenty reciprocal complexes that are not homology manifolds as there are plenty of semi-Eulerian complexes that are not homology manifolds. For instance, consider the ``double banana'' space, \ie two 2-dimensional spheres that are glued along two points. (As an example, take two copies of the boundary complex of octahedron and attach them by identifying two opposite vertices of the first copy with two opposite vertices of the second.) This space is even Eulerian, but it is definitely not a homology manifold as the link of each of the two special points is the union of two circles.

Most importantly, while all homology spheres, all odd-dimensional closed homology manifolds\footnote{A closed homology manifold refers to a homology manifold without boundary.} as well as all even-dimensional closed homology manifolds of characteristic two are Eulerian, there are plenty of other Eulerian complexes. Similarly, while all closed homology manifolds are semi-Eulerian, there are plenty of other semi-Eulerian complexes.''
\end{quotation}

\medskip
In the same spirit, we remark that there are plenty of reciprocal complexes which are not semi-Eulerian. As an explicit example, remove a triangle from the ``double banana'' described above. In this case, the multiplicity $m_F$ of any vertex or edge of this removed triangle is equal to zero, while the multiplicity of all other non-empty faces is equal to one.

The second question is related to the structure of boundary faces of reciprocal complexes. Recall that a non-empty face $F\in \SC$ is called a boundary face if $m_\emptyset =0$. It is not clear from this definition whether the collection of boundary faces gives a sub-complex. 

\begin{question}
	Is the collection of boundary faces of a reciprocal complex $\SC$ a sub-complex? In other words, is the face of a boundary face also a boundary face? 
	If so, is the boundary $\partial \SC$ a codimension~1 sub-complex? Is it reciprocal? 
\end{question}  

The subtle study of the boundary complex of homology manifolds is investigated in~\cite{grabe_uber_1984} (see also~\cite{grabe_DehnSommerville_1987}) and~\cite{mitchell_boundaryHomologyManifolds_1990}. Quoting~\cite{grabe_DehnSommerville_1987}:
\begin{quotation}
``
... nothing can be said about the behaviour of this boundary if~$\SC$ is a non-orientable homology $N$-manifold. 
Thus let us call $Bd\,\SC$ a g\underline{ood boundar}y if it is a homology $(N-1)$-manifold without boundary.
All finite triangulations of topological manifolds are so. 
''
\end{quotation}

\section{General simplicial complexes}
In this section, we present the $h$-version of the Dehn--Sommerville relations for arbitrary abstract simplicial complexes as recently shown by Sawaske and Xue in~\cite{sawaske_nonEulerianDehnSommerville_2021}. As a corollary, one obtains the $h$-version of Klee's Dehn--Sommerville relations for semi-Eulerian complexes~\cite{klee_DehnSommervilleRelations_1964}.   

As before, these results are straightforward consequences of the two evaluations~\eqref{eq_fh_tilde} and~\eqref{eq_fh_reciprocity} of the $\tilde h$-polynomial in Theorem~\ref{thm_fh_tilde} and Theorem~\ref{thm_fh_reciprocity}, respectively.    
We highlight that these two theorems are the essence of the Dehn--Sommerville relations. 

Following~\cite{sawaske_nonEulerianDehnSommerville_2021}, we define the error $\epsilon_F$ of a face $F\in \SC$ as
\begin{align}\label{eq_error_definition1}
\epsilon_F \defs \widetilde \chi\bigl(\link{F}{\SC}\bigr) - (-1)^{d-1-|F|},
\end{align}
or equivalently,
\begin{align}\label{eq_error_definition2}
\epsilon_F = (-1)^{d-1-|F|} (m_F-1).
\end{align}
 
Note that $\link{F}{\SC}$ has dimension $d-1-|F|$, and $(-1)^{d-1-|F|}$ is the reduced Euler characteristic of a sphere of this dimension. 

\begin{theorem}[Polynomial Dehn--Sommerville relation for simplicial complexes, $h$-version]
\label{thm_PolynomialDehnSommerville_simplicialComplexes}
Let $\SC$ be a simplicial complex of dimension $d-1$. 
The following relation holds:
\begin{align}\label{eq_polynomialDehnSommerville_hversion}
\sum_{i=0}^d (h_i-h_{d-i})(x+1)^ix^{d-i} = 
\sum_{F\in \SC} (m_F-1)x^{|F|}.
\end{align}
\end{theorem}

\begin{proof}
The left hand side of~\eqref{eq_polynomialDehnSommerville_hversion} is equal to the difference between two evaluations of the $\tilde h$-polynomial:
\begin{align*}
x^d\tilde h\bigl( \frac{x+1}{x} \bigr) - 
(x+1)^d\tilde h\bigl( \frac{x}{x+1} \bigr).
\end{align*}
The result then follows from Equations~\eqref{eq_fh_reciprocity} and~\eqref{eq_fh_tilde}.
\end{proof}

The following corollary was proven in~\cite[Theorem~3.1]{sawaske_nonEulerianDehnSommerville_2021}\footnote{In \cite[Theorem~3.1]{sawaske_nonEulerianDehnSommerville_2021}, the result is stated only for pure simplicial complexes. The same result holds for non-pure complexes as stated in Corollary~\ref{cor_DehnSommerville_simplicialComplexes}.}.

\begin{corollary}[{\cite{sawaske_nonEulerianDehnSommerville_2021}} Dehn--Sommerville relations for simplicial complexes, $h$-version]
\label{cor_DehnSommerville_simplicialComplexes}
Let $\SC$ be a simplicial complex of dimension $d-1$. 
For $i=0,1,\dots,d$, the following relations hold:
\begin{align}\label{eq_DehnSommerville_simplicialComplexes}
h_{d-i}-h_i = (-1)^i \sum_{F\in \SC} {d-|F| \choose i} \epsilon_F.
\end{align}
\end{corollary}

\begin{proof}
This result is a reformulation of the polynomial Dehn--Sommerville relation~\eqref{eq_polynomialDehnSommerville_hversion}.
As pointed out in Appendix~\ref{appendix_polynomials} (Proposition~\ref{prop_polynomialBases}), the collection $$\bigl\{(x+1)^ix^{d-i}\colon 0\leq i \leq d\bigr\}$$ forms a basis for the vector space of polynomials of degree less than or equal to~$d$. Thus, we just need to expand $x^{|F|}$ in this basis, and extract the coefficient of \mbox{$(x+1)^ix^{d-i}$}. By~\eqref{eq_changeBasis}, this coefficient is equal to
\begin{align*}
(-1)^{d-|F|-i} {d-|F| \choose i}.
\end{align*} 
Multiplying by $(m_F-1)$ and taking the sum over all $F\in \SC$ we get
\begin{align*}
h_i-h_{d-i} 
&= \sum_{F\in \SC} (m_F-1)  (-1)^{d-|F|-i} {d-|F| \choose i} \\
&= - \sum_{F\in \SC} (-1)^{i} {d-|F| \choose i} \epsilon_F. \quad \quad (\text{by}~\eqref{eq_error_definition2})
\end{align*}
This is equivalent to~\eqref{eq_DehnSommerville_simplicialComplexes} as desired.
\end{proof}

\begin{corollary}[{\cite{klee_DehnSommervilleRelations_1964}} Dehn--Sommerville relations for semi-Eulerian complexes, $h$-version]
\label{cor_eq_DehnSommerville_semiEulerian_hversion}
Let $\SC$ be a semi-Eulerian complex of dimension $d-1$. 
For $i=0,1,\dots,d$, the following relations hold:
\begin{align}\label{eq_eq_DehnSommerville_semiEulerian_hversion}
h_{d-i}-h_i = (-1)^i {d \choose i} \bigl( \widetilde \chi(\SC) - (-1)^{d-1}\bigr).
\end{align}
In particular, if $\SC$ is Eulerian then $h_{d-i}=h_i$.
\end{corollary}

\begin{proof}
If $\SC$ is semi-Eulerian then $\epsilon_F=0$ for every non-empty face. So, the only term that survives in the sum~\eqref{eq_DehnSommerville_simplicialComplexes} corresponds to $F=\emptyset$, and
\begin{align*}
h_{d-i}-h_i 
= (-1)^i{d \choose i} \epsilon_\emptyset 
= (-1)^i{d \choose i} \bigl( \widetilde \chi\bigl(\SC \bigr) - (-1)^{d-1}\bigr).
\end{align*} 
In addition, if $\SC$ is Eulerian then $\widetilde \chi\bigl(\SC \bigr) - (-1)^{d-1} = 0$ and $h_{d-i}=h_i$.
\end{proof}

\section{Balanced simplicial complexes}
In this section we revisit several versions of the Dehn--Sommerville relations for balanced simplicial complexes, a concept introduced by Stanley in~\cite{stanley_balanced_1987}. We present a new version in Theorem~\ref{thm_DehnSommerville_balancedSimplicialComplexes}, which 
\begin{itemize}
\item recovers the Dehn--Sommerville relations for balanced semi-Eulerian complexes by Swartz~\cite{swartz_face_enumeration_2009} (Corollary~\ref{cor_eq_DehnSommerville_balancedSemiEulerian_hversion}), as well as the specialization for completely balanced Eulerian complexes by Bayer and Billera in~\cite{bayer_generalized_1985}; and
\item generalizes the recent Dehn--Sommerville relations for completely balanced simplicial complexes by Sawaske and Xue~\cite{sawaske_nonEulerianDehnSommerville_2021} (Corollary~\ref{cor_DehnSommerville_completelyBalancedSimplicialComplexes}). 
\end{itemize}

Let $\N=\{0,1,2,\dots\}$ be the set of non-negative integers, $m$ be a positive integer and $\bolda=(a_1,\dots,a_m)\in \N^m$. 
A simplicial complex $\SC$ of dimension $d-1$ with vertex set $V$ is called \defn{balanced of type $\bolda$}, if it is equipped with a vertex coloring 
$\kappa: V \rightarrow [m]$ such that for every maximal face $F\in \SC$ the number of vertices in $F$ with color $i$ is equal to $a_i$.
Hence, $|\bolda|\defs a_1+\dots +a_m = d$ and $\SC$ is pure.  
Balanced complexes of type $(1,1,\dots,1)$ are called \defn{completely balanced}.

For $F\in \SC$ we denote by $\boldb(F)=(b_1,\dots,b_m)$ the vector whose entries count the number of vertices in $F$ with color $i$, that is
\[
b_i = \bigl\lvert\{v\in F\colon \kappa(v)=i\}\bigr\rvert. 
\]
The \defn{flag $f$-vector} of $\SC$ is the collection $\bigl( f_\boldb \bigr)_{\boldb\leq \bolda}$ where $\boldb\in\N^m$ and $f_\boldb$ counts the number of faces $F$ in $\SC$ such that $\boldb(F)=\boldb$. The numbers $f_\boldb$ are called the \defn{flag $f$-numbers}. The $\defn{flag $\tilde f$-polynomial}$ $\tilde f(\x)\in \Z[x_1,\dots ,x_m]$ is a polynomial in $m$ variables $\x\defs (x_1,\dots, x_m)$ defined by
\begin{align}
\tilde f(\x) \defs \sum_{F\in \SC} \x^{\boldb(F)} 
= \sum_{\boldb\leq \bolda} f_\boldb \x^\boldb,
\end{align}
where $\x^\boldb\defs \prod_{i=1}^m x_i^{b_i}$ for any $\boldb \in \N^m$. In general and for simplicity, given any univariate
function $p(x)$ we will denote by $p(\x)$ the multivariate function
\begin{align}
p(\x)^\boldb = \prod_{i=1}^m p(x_i)^{b_i}.
\end{align} 

Similarly as in~\eqref{relation_fh_vectors}, we define the \defn{flag $\tilde h$-polynomial} by
\begin{equation}\label{relation_flag_fh_vectors}
	\tilde h (\x)
	\defs
	\sum_{F\in \SC}  \x^{\boldb(F)}(\mathbf{1}-\x)^{\bolda - \boldb(F)}.  
\end{equation}

This is a multivariate polynomial of degree less than or equal to $\bolda$. Thus, it can be written uniquely as a linear combination of the form
\begin{align}
\tilde h (\x) \defs
\sum_{\boldb \leq \bolda} h_\boldb \x^\boldb.
\end{align} 

The numbers $h_\boldb$ are called the \defn{flag $h$-numbers}. It is not hard to check that they can be computed by 
\begin{align}
h_\boldb = \sum_{\boldc\leq \boldb} (-1)^{|\boldb|-|\boldc|}f_\boldc.
\end{align}

The following two theorems are direct multivariate generalizations of Theorem~\ref{thm_fh_tilde} and Theorem~\ref{thm_fh_reciprocity}. Their proofs are essentially the same, we include them here for completeness.  

\begin{theorem}\label{thm_fh_tilde_multi}
Let $\SC$ be a balanced simplicial complex of type $\bolda$
with flag $\tilde f$-polynomial  $\tilde f(\x)$ and flag $\tilde h$-polynomial  $\tilde h(\x)$. 
The following relation holds:
\begin{align}\label{eq_fh_tilde_multi}
	 (\x+\mathbf{1})^{\bolda} \tilde h\left(\frac{\x}{\x+\mathbf{1}}\right) = \tilde f(\x),
\end{align}
where $\frac{\x}{\x+\mathbf{1}} \defs \left(\frac{x_1}{x_1+1},\dots,\frac{x_d}{x_d+1} \right)$.
\end{theorem}
\begin{proof}
Replacing $\x$ by $\frac{\x}{\x+\mathbf{1}}$ in~\eqref{relation_flag_fh_vectors}, and multiplying the result by $(\x+\mathbf{1})^{\bolda}$ we get
\begin{align*}
(\x+\mathbf{1})^{\bolda} \tilde h\left(\frac{\x}{\x+\mathbf{1}}\right) 
= (\x+\mathbf{1})^{\bolda} \sum_{F\in\SC}  \left(\frac{\x}{\x+\mathbf{1}}\right)^{\boldb(F)}\left( \frac{\mathbf{1}}{\x+\mathbf{1}}\right)^{\bolda - \boldb(F)}
= \sum_{F\in\SC} \x^{\boldb(F)}
= \tilde f(\x),
\end{align*}
as desired.
\end{proof}

\begin{theorem}[The $f{=}h$ reciprocity for balanced simplicial complexes]
\label{thm_fh_reciprocity_multi}
Let $\SC$ be a balanced simplicial complex of type $\bolda$ with flag $\tilde h$-polynomial  $\tilde h(\x)$ and $d=|\bolda|$. Then 
\begin{equation}\label{eq_fh_reciprocity_multi}
 \x^{\bolda} \tilde h\left(\frac{\x+\mathbf{1}}{\x}\right) =
  \sum_{F\in \SC} m_F \x^{\boldb(F)},  
\end{equation}
where $\frac{\x+\mathbf{1}}{\x} \defs \left(\frac{x_1+1}{x_1},\dots,\frac{x_d+1}{x_d} \right)$ and 
$m_F = \sum_{F\subseteq G \in \SC} (-1)^{d-|G|}  = (-1)^{d-1-|F|}\widetilde \chi\bigl(\link{F}{\SC}\bigr)$.
\end{theorem}

Note that the multiplicity $m_F$ here is the same as the one used in Theorem~\ref{thm_fh_reciprocity}.

\begin{proof}
From a direct substitution in~\eqref{relation_flag_fh_vectors} we get:
\begin{align*}
 \x^{\bolda} \tilde h\left(\frac{\x+\mathbf{1}}{\x}\right) 
 &=  \sum_{G\in\SC} (\x+\mathbf{1})^{\boldb(G)} (-1)^{|\bolda|-|\boldb(G)|} \\
 &=  \sum_{G\in \SC} \left( \sum_{F\subseteq G} \x^{\boldb(F)} \right) (-1)^{d-|G|}\\
 &= \sum_{F\in \SC} \x^{\boldb(F)} \left( \sum_{F\subseteq G\in \SC} (-1)^{d-|G|} \right) \\
 &= \sum_{F\in \SC} m_F \x^{\boldb(F)}.
\end{align*}
This finishes our proof.
\end{proof}

Similarly as before, Theorem~\ref{thm_fh_tilde_multi} and Theorem~\ref{thm_fh_reciprocity_multi} lie at the heart of the Dehn--Sommerville relations for balanced simplicial complexes. The key stone is the following polynomial relation.

\begin{theorem}[Polynomial Dehn--Sommerville relation for balanced simplicial complexes, $h$-version]
\label{thm_Polynomial_balancedSimplicialComplexes}
Let $\SC$ be a balanced simplicial complex of type~$\bolda$. 
The following relation holds:
\begin{align}\label{eq_polynomialDehnSommerville_balanced_hversion}
\sum_{\boldb \leq \bolda} (h_\boldb-h_{\bolda-\boldb})\x^\boldb(\x+\mathbf{1})^{\bolda-\boldb} = 
\sum_{F\in \SC} (1-m_F)\x^{\boldb(F)}.
\end{align}
\end{theorem}

\begin{proof}
The left hand side of~\eqref{eq_polynomialDehnSommerville_balanced_hversion} is equal to the difference between two evaluations of the flag $\tilde h$-polynomial\footnote{Note that we are reversing the order of the two evaluations compared to the proof of Theorem~\ref{thm_PolynomialDehnSommerville_simplicialComplexes} for convenience.}:
\begin{align*}
(\x+\mathbf{1})^{\bolda}\tilde h\bigl( \frac{\x}{\x+\mathbf{1}} \bigr) -
\x^{\bolda}\tilde h\bigl( \frac{\x+\mathbf{1}}{\x} \bigr).
\end{align*}
The result then follows from Equations~\eqref{eq_fh_tilde_multi} and~\eqref{eq_fh_reciprocity_multi}.
\end{proof}

Recall that the error of a face is defined as $\epsilon_F = (-1)^{d-1-|F|} (m_F-1)$, where \mbox{$d=|\bolda|$} for a balanced simplicial complex of type $\bolda$.
 As a consequence of the previous theorem, we have the following result which is new in this general form. Here we use the multi-binomial coefficient notation 
${\boldb \choose \boldc} \defs \prod_{i=1}^m {b_i \choose c_i}$ for~$\boldb,\boldc \in \N^m$.

\begin{theorem}[Dehn--Sommerville relations for balanced simplicial complexes, $h$-version]
\label{thm_DehnSommerville_balancedSimplicialComplexes}
Let $\SC$ be a balanced simplicial complex of type~$\bolda$. 
For $\boldb \leq \bolda$, the following relations hold:
\begin{align}\label{eq_DehnSommerville_balancedSimplicialComplexes}
h_{\boldb}-h_{\bolda-\boldb} = (-1)^{|\bolda|-|\boldb|} \sum_{F\in \SC_\boldb} {\bolda-\boldb(F) \choose \bolda-\boldb} \epsilon_F,
\end{align}
where $\SC_\boldb \defs \{F\in\SC\colon\boldb(F)\leq \boldb\}$ denotes the set of faces with at most $b_i$ vertices with color $i$ for all $i$.
\end{theorem}

\begin{proof}
This result is a reformulation of the polynomial Dehn--Sommerville relation~\eqref{eq_polynomialDehnSommerville_balanced_hversion}.
As pointed out in Appendix~\ref{appendix_polynomials} (Proposition~\ref{prop_polynomialBases_multi}), the collection 
$$\{\x^\boldb(\x+\mathbf{1})^{\bolda-\boldb}: \boldb\leq \bolda\}$$ 
forms a basis for the vector space of polynomials of degree less than or equal to~$\bolda$ in $\Z[x_1,\dots,x_m]$. Thus, we just need to expand $\x^{\boldb(F)}$ in this basis, and extract the coefficient of $\x^\boldb(\x+\mathbf{1})^{\bolda-\boldb}$. 

By~\eqref{eq_changeBasis_multi}, this coefficient is equal to zero if~$\boldb(F)\nleq \boldb$. 
If~$\boldb(F)\leq \boldb$ , or equivalently $F\in\SC_\boldb$, this coefficient is equal to
\begin{align*}
(-1)^{|\boldb|-|\boldb(F)|} {\bolda-\boldb(F) \choose \bolda-\boldb}
= (-1)^{|\boldb|-|F|} {\bolda-\boldb(F) \choose \bolda-\boldb}.
\end{align*} 
Multiplying by $(1-m_F)$ and taking the sum over $F\in \SC_\boldb$ we get by~\eqref{eq_polynomialDehnSommerville_balanced_hversion} that
\begin{align*}
h_{\boldb}-h_{\bolda-\boldb} 
&=  \sum_{F\in \SC_\boldb} (1-m_F)(-1)^{|\boldb|-|F|} {\bolda-\boldb(F) \choose \bolda-\boldb} \\
&=  \sum_{F\in \SC_\boldb} (-1)^{d+|\boldb|-2|F|} \epsilon_F {\bolda-\boldb(F) \choose \bolda-\boldb}\\
&= (-1)^{d-|\boldb|} \sum_{F\in \SC_\boldb} {\bolda-\boldb(F) \choose \bolda-\boldb} \epsilon_F.
\end{align*}
Since $d=|\bolda|$, this finishes our proof.
\end{proof}

The specialization stated in Corollary~\ref{cor_eq_DehnSommerville_balancedSemiEulerian_hversion} below, recovers the Dehn--Sommerville relations for balanced semi-Eulerian complexes proven by Swartz in~\cite[Theorem~3.8]{swartz_face_enumeration_2009}. The sub-case of completely balanced spheres (or in general, completely balanced Eulerian complexes) was presented earlier by Bayer and Billera in~\cite[Section~3]{bayer_generalized_1985}. The result for completely balanced semi-Eulerian posets appears in the work of Stanley in~\cite[Proposition~2.2]{stanley_some_1982}. We refer to~\cite{bayer_generalized_1985} for more details on the history of these relations. 

\begin{corollary}[{\cite{swartz_face_enumeration_2009}} Dehn--Sommerville relations for balanced semi-Eulerian complexes, $h$-version]
\label{cor_eq_DehnSommerville_balancedSemiEulerian_hversion}
Let $\SC$ be a balanced semi-Eulerian complex of type~$\bolda$. 
For $\boldb\leq \bolda$, the following relations hold:
\begin{align}\label{eq_eq_DehnSommerville_balancedSemiEulerian_hversion}
h_{\bolda-\boldb}-h_\boldb = (-1)^{|\boldb|} \bigl( \widetilde \chi(\SC) - (-1)^{d-1}\bigr) {\bolda \choose \boldb}.
\end{align}
In particular: 
\begin{itemize}
\item if $\SC$ is Eulerian then $h_{\bolda-\boldb}=h_{\boldb}$;
\item if $\SC$ is completely balanced ($\bolda=(1,1,\dots,1)$), then 
${\bolda \choose \boldb}=1$.   
\end{itemize}
\end{corollary}

\begin{proof}
As in the proof of Corollary~\ref{cor_eq_DehnSommerville_semiEulerian_hversion}, the result follows from~\eqref{eq_DehnSommerville_balancedSimplicialComplexes} and the fact that $\epsilon_F=0$ for every non-empty face $F\in \SC$ and $\epsilon_\emptyset=\widetilde \chi(\SC) - (-1)^{d-1}$. We get
\begin{align*}
h_{\boldb}-h_{\bolda-\boldb} = (-1)^{|\bolda|-|\boldb|}  {\bolda \choose \bolda-\boldb} (\widetilde \chi(\SC) - (-1)^{d-1}),
\end{align*}
In order to obtain~\eqref{eq_eq_DehnSommerville_balancedSemiEulerian_hversion} we evaluate the previous equation at $\boldb=\bolda-\boldb$.  
The remaining two items are straight forward from~\eqref{eq_eq_DehnSommerville_balancedSemiEulerian_hversion}.
\end{proof}

Theorem~\ref{thm_DehnSommerville_balancedSimplicialComplexes} is a generalization of a recent result of Sawaske and Xue~\cite[Theorem~4.1]{sawaske_nonEulerianDehnSommerville_2021}. Their result holds for completely balanced simplicial complexes and is stated without using the multi-binomial coefficient\footnote{In~\cite{sawaske_nonEulerianDehnSommerville_2021}, the result is stated for balanced simplicial complexes. However, the definition used by the authors is that of completely balanced simplicial complexes.}. 

\begin{corollary}[{\cite{sawaske_nonEulerianDehnSommerville_2021}}]
\label{cor_DehnSommerville_completelyBalancedSimplicialComplexes}
Let $\SC$ be a balanced simplicial complex of type~$\bolda=(1,1,\dots,1)$ (\ie a completely balanced complex). 
For $\boldb \leq \bolda$, the following relations hold:
\begin{align}\label{eq_DehnSommerville_completelyBalancedSimplicialComplexes}
h_{\boldb}-h_{\bolda-\boldb} = (-1)^{|\bolda|-|\boldb|} \sum_{F\in \SC_\boldb} \epsilon_F.
\end{align}
\end{corollary}

\begin{proof}
Note that if $\boldc\in \N^m$ consists of only zeros and ones and $\boldc'\leq \boldc$ then ${\boldc \choose \boldc'}=1$. 
Now, let $\bolda=(1,1,\dots ,1)$ (completely balanced condition). For every $F\in\SC_\boldb$ we have $\boldb(F)\leq \boldb$, and so taking $\boldc=\bolda-\boldb(F)$ and $\boldc'=\bolda-\boldb$ we have ${\boldc \choose \boldc'}=1$.
The result then follows directly from Theorem~\ref{thm_DehnSommerville_balancedSimplicialComplexes}.
\end{proof}

\section{Examples of the $f=h$ reciprocity}\label{sec_examples}
All instances of the Dehn--Sommerville relations presented in this paper are consequences of the $f{=}h$ reciprocity described in Theorem~\ref{thm_fh_reciprocity} (and its multivariate version in Theorem~\ref{thm_fh_reciprocity_multi}). 
Figure~\ref{fig_examples_fh_reciprocity} illustrates four examples of this result:

\begin{figure}
	\begin{subfigure}[b]{.45\textwidth}
		\centering
		\includegraphics[scale=1,page=12]{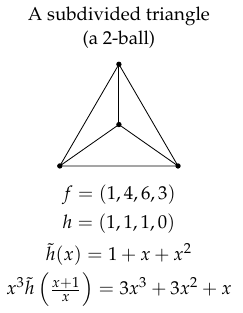}
	\end{subfigure}
	\hspace*{.5cm}
	\begin{subfigure}[b]{.45\textwidth}
		\centering
		\includegraphics[scale=1,page=13]{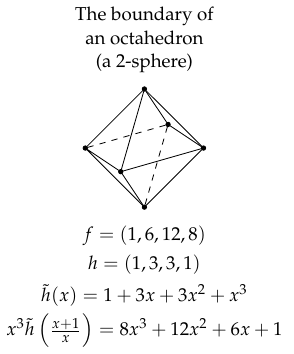}
	\end{subfigure}
	
	\vspace*{.5cm}
	
	\begin{subfigure}[t]{.45\textwidth}
		\centering
		\includegraphics[scale=1,page=14]{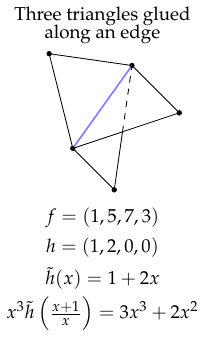}
	\end{subfigure}
	\hspace*{.5cm}
	\begin{subfigure}[t]{.45\textwidth}
		\centering
		\includegraphics[scale=1,page=15]{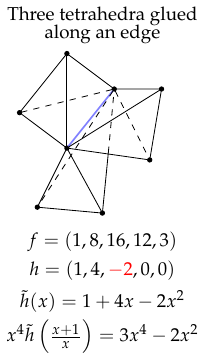}
	\end{subfigure}
\caption{Four examples of the $f{=}h$ reciprocity.}
\label{fig_examples_fh_reciprocity}
\end{figure}

(1) The first example is the simplicial complex of a triangle which is subdivided in three smaller triangles. 
This is a 2-dimensional ball, and the evaluation 
$ x^d \tilde h\left(\frac{x+1}{x}\right) = 3x^3+3x^2+1x$ enumerates its interior faces (those with multiplicity~$m_F=1$): 
3 triangles, 3 interior edges, and 1 interior vertex.
   
(2) The second example is the boundary complex of an octahedron. 
This is a 2-dimensional sphere and therefore also an Eulerian complex. The evaluation    
$ x^d \tilde h\left(\frac{x+1}{x}\right) = 8x^3+12x^2+6x+1$ enumerates all its faces (all of them are interior and have multiplicity $m_F=1$):
8 triangles, 12 edges, 6 vertices, and 1 empty face. 

(3) The third example is the complex of three triangles glued along an edge.
The evaluation 
$ x^d \tilde h\left(\frac{x+1}{x}\right) = 3x^3+2x^2$ 
enumerates the faces $F$ with multiplicity~$m_F$:
3~triangles with multiplicity 1, and 1 edge with multiplicity 2. 
This last multiplicity comes from the common edge of the three triangles, whose link consists of three disconnected vertices and has reduced Euler characteristic 2. All other faces have multiplicity zero.

(4) The fourth example is the complex of three tetrahedra glued along an edge. 
The evaluation 
$ x^d \tilde h\left(\frac{x+1}{x}\right) = 3x^4-2x^2$ 
enumerates the faces $F$ with multiplicity~$m_F$:
3 tetrahedra with multiplicity 1, and 1 edge with multiplicity $-2$. 
This last multiplicity comes from the common edge of the three tetrahedra, whose link consists of three disconnected segments and has reduced Euler characteristic~2. The multiplicity of this edge is then $(-1)^{3-2} \times 2=-2$. This is an interesting case because we have a negative multiplicity. All other faces have multiplicity zero.

\appendix
\section{Macdonald's polynomial Dehn--Sommerville relation}\label{appendix_Macdonald}
In his 1971 paper~\cite{Macdonald_DehnSommerville_1971}, Macdonald proved a polynomial Dehn--Sommerville relation for a finite simplicial complex $\SC$ whose underlying topological space is a $(d-1)$-dimensional manifold. As far as we know, this is the first appearance of a Dehn--Sommerville type relation for manifolds (possibly with boundary), after Klee's Dehn--Sommerville relations for semi-Eulerian complexes from 1964~\cite{klee_DehnSommervilleRelations_1964}.  

In this appendix, we recall Macdonald's polynomial relation, and show that it follows from the polynomial Dehn--Sommerville relation~\eqref{eq_f_fint_one} in Theorem~\ref{thm_polynomialDehnSommerville_homology}. We also explain why Macdonald's relation is slightly weaker.  

Throughout this appendix, we fix a simplicial complex $\SC$ whose underlying topological space is a $(d-1)$-dimensional manifold, and denote by $f_i$, $f_i^\interior$, and~$f_i^\partial$ the number of $i$-dimensional faces, interior faces, and boundary faces, respectively. 
Macdonald considered the polynomial 
\begin{align}
\widetilde P(\SC,x) \defs 1-f_0x+f_1x^2-\dots + (-1)^df_{d-1}x^d,
\end{align}
and proved the following result.\footnote{The term $-\widetilde \chi (\SC)$ on the right hand side of~\eqref{eq_Macdonald_one} differs from $\widetilde \chi (\SC)$ in~\cite[Theorem~2.1]{Macdonald_DehnSommerville_1971} in a minus sign. This is not a mistake, because in~\cite[Page 182]{Macdonald_DehnSommerville_1971} the value of $\widetilde \chi (\SC)$ is defined as the ``augmented Euler characteristic'' which is equal to $-\widetilde \chi (\SC)$ as defined here.} 

\begin{theorem}[{\cite[Theorem~2.1]{Macdonald_DehnSommerville_1971}} Macdonald polynomial Dehn--Sommerville relation for manifolds]
The polynomial $Q(x)\defs \widetilde P(\SC,x)-\frac{1}{2} \widetilde P(\partial \SC,x)$ satisfies
\begin{align}\label{eq_Macdonald_one}
Q(1-x)+(-1)^{d-1}Q(x)=
\begin{cases}
0, & \text{ if } d-1 \text{ is odd},\\
-\widetilde \chi (\SC), & \text{ if } d-1 \text{ is even.}
\end{cases}
\end{align}
\end{theorem}

In order to compare this with the polynomial relation~\eqref{eq_f_fint_one} in Theorem~\ref{thm_polynomialDehnSommerville_homology}, it is convenient to replace $x$ by $-x$ and rewrite~\eqref{eq_Macdonald_one} as:

\begin{align}\label{eq_Macdonald_two}
(-1)^dQ(-x) = Q(1+x) + \text{cte},
\end{align}
where the constant term is
\begin{align}
\text{cte} = 
\begin{cases}
0, & \text{ if } d-1 \text{ is odd},\\
\widetilde \chi (\SC), & \text{ if } d-1 \text{ is even.}
\end{cases}
\end{align}

Now note that 
\begin{align}
\widetilde P(\SC,-x) &= 1+f_0x+f_1x^2+\dots + (-1)^df_{d-1}x^d  = \tilde f(x), \\
\widetilde P(\partial \SC,-x) &=  1+f_0^\partial x+f_1^\partial x^2+\dots + (-1)^df_{d-1}^\partial x^d.
\end{align}

Therefore, we have $\widetilde P(\SC,-x)-\widetilde P(\partial \SC,-x)=\tilde f^\interior(x)$, and so
\begin{align}
Q(-x) &= \frac{1}{2} \left( \tilde f(x) + \tilde f^\interior(x) \right), \\
Q(1+x) &=  \frac{1}{2} \left( \tilde f(-x-1) + \tilde f^\interior(-x-1) \right).
\end{align}

Now recall the polynomial Dehn--Sommerville relation~\eqref{eq_f_fint_one} in Theorem~\ref{thm_polynomialDehnSommerville_homology}:
\begin{align*}
(-1)^d \tilde f(x) = \tilde f^\interior (-x-1) + m_\emptyset.
\end{align*}
Here, the multiplicity is $m_\emptyset=(-1)^{d-1}\widetilde \chi (\SC)$. Replacing $x$ by $-x-1$ in the previous equation gives an alternative version of the polynomial Dehn--Sommerville relation, which was already stated in Equation~\eqref{eq_f_fint_two}:  
\begin{align*}
  (-1)^d \tilde f^\interior (x)  = \tilde f(-x-1) - (-1)^d m_\emptyset.
\end{align*}

By adding Equations~\eqref{eq_f_fint_one} and~\eqref{eq_f_fint_two}, and multiplying the result by $\frac{1}{2}$, we get exactly Macdonald's polynomial relation~\eqref{eq_Macdonald_two}. We have just proved:

\begin{lemma}
The polynomial Dehn--Sommerville relation~\eqref{eq_f_fint_one} implies Macdonald's polynomial relation~\eqref{eq_Macdonald_two}.
\end{lemma}

On the other hand, since~\eqref{eq_Macdonald_two} is (half of) the sum of~\eqref{eq_f_fint_one} and~\eqref{eq_f_fint_two} (which is equivalent to~\eqref{eq_f_fint_one}), Equation~\eqref{eq_f_fint_one} can not be deduced from~\eqref{eq_Macdonald_two}.

\begin{lemma}
Macdonald's polynomial relation~\eqref{eq_Macdonald_two} does not imply the polynomial Dehn--Sommerville relation~\eqref{eq_f_fint_one}.
\end{lemma}

\begin{proof}
For a fixed $d$ and fixed reduced Euler characteristic $\widetilde \chi (\SC)$, the polynomial Dehn--Sommerville relation~\eqref{eq_f_fint_one} completely determines the face numbers from the interior face numbers (see~\eqref{eq_gen_DehnSommerville_one}) and vice versa (see~\eqref{eq_gen_DehnSommerville_version_two}).
We claim that this is not true for Macdonald's polynomial relation~\eqref{eq_Macdonald_two}. 
For this, it suffices to look at a small example. 

Fix $d=3$ and $\widetilde \chi (\SC)=-1$. Consider the $f$-vector $f=(1,4,8,4)$ of the triangulated cylinder in Example~\ref{example_cylinder}.
Comparing the coefficients of the polynomial relation~\eqref{eq_f_fint_one} completely determines the interior face numbers $f^\interior=(\cdot,0,4,4)$, as well as the boundary face numbers $f^\partial=(1,4,4,0)$.

On the other hand, $f^\partial=(1,5,7,2)$ is also a solution to Macdonald's polynomial relation~\eqref{eq_Macdonald_two}. Thus, Equation~\eqref{eq_Macdonald_two} can not imply~\eqref{eq_f_fint_one}.
\end{proof}

We highlight that the alternative solution $f^\partial=(1,5,7,2)$ in the previous proof is not realistic, because one can not have more boundary vertices ($f_0^\partial=5$) than vertices itself ($f_0=4$). But the linear relations obtained by comparing the coefficients in Macdonald's relation~\eqref{eq_Macdonald_two} allow such examples, while~\eqref{eq_f_fint_one} does not. 

Building a bit more on this example, for $d=3$ we have
\begin{align*}
\tilde f(x) &= 1+f_0x+f_1x^2+f_2x^3, \\
\tilde f^\interior(x) &= f_0^\interior x+f_1^\interior x^2+f_2^\interior x^3.
\end{align*}
Comparing the coefficient of $x^k$, for $k>0$, in the polynomial relation~\eqref{eq_f_fint_one}, gives the following system of linear relations:
\begin{align}
 f_{0} &=  f_0^\interior -2 f_1^\interior+3 f_2^\interior, \label{eq_example_d_3_one}\\
 f_{1} &=  -f_1^\interior+3f_2^\interior, \label{eq_example_d_3_two}\\
 f_{2} &= f_2^\interior, \label{eq_example_d_3_three}
\end{align}
while for $k=0$, we obtain 
\begin{align}
f_0-f_1+f_2&=f_0^\interior-f_1^\interior+f_2^\interior \label{eq_example_d_3_four}. 
\end{align}
We can do the same game with the polynomial relation~\eqref{eq_f_fint_two}, which is equivalent to~\eqref{eq_f_fint_one}. 
We obtain:
\begin{align}
 f_{0}^\interior &=  f_0 -2 f_1+3 f_2, \label{eq_example_d_3_versiontwo_one}\\
 f_{1}^\interior &=  -f_1+3f_2, \label{eq_example_d_3_versiontwo_two}\\
 f_{2}^\interior &= f_2. \label{eq_example_d_3_versiontwo_three}
\end{align}
Comparing the constant term (coefficient of $x^0$) in~\eqref{eq_f_fint_two} just leads to the irrelevant equation 
\begin{align}
 0=0. \label{eq_example_d_3_versiontwo_four}
\end{align}

In summary, the systems of linear relations~\eqref{eq_example_d_3_one}--\eqref{eq_example_d_3_four} and~\eqref{eq_example_d_3_versiontwo_one}--\eqref{eq_example_d_3_versiontwo_four} follow from the polynomial Dehn--Sommerville relation~\eqref{eq_f_fint_two}. 
Macdonald's polynomial relation~\eqref{eq_Macdonald_two} only implies the following linear relations:

\begin{align}
 f_{0}+ f_{0}^\interior &=  (f_0^\interior -2 f_1^\interior+3 f_2^\interior)+ (f_0 -2 f_1+3 f_2), \label{eq_example_d_3_Macdonald_one}\\
 f_{1}+f_{1}^\interior  &=  (-f_1^\interior+3f_2^\interior)+(-f_1+3f_2), \label{eq_example_d_3_Macdonald_two}\\
 f_{2}+f_{2}^\interior &= (f_2^\interior)+(f_2), \label{eq_example_d_3_Macdonald_three} \\
 f_0-f_1+f_2&=f_0^\interior-f_1^\interior+f_2^\interior \label{eq_example_d_3_Macdonald_four}. 
\end{align}

Equation~\eqref{eq_example_d_3_Macdonald_three} becomes $0=0$ and vanishes, while Equations~\eqref{eq_example_d_3_Macdonald_one} and~\eqref{eq_example_d_3_Macdonald_two} are equivalent. This system gets reduced to only two linear relations:
\begin{align}
 2f_{1}+2f_{1}^\interior  &=  3f_2^\interior+3f_2, \\
 f_0-f_1+f_2&=f_0^\interior-f_1^\interior+f_2^\interior . 
\end{align}

For $f=(1,4,8,4)$ we have several possible solutions for the interior face numbers. In particular, $f^\interior=(\cdot,-1,1,2)$ works (where $f_2= f_2^\interior$ is not satisfied, a necessary relation which is forgotten by Macdonald's relation). This is how we found the solution $f^\partial=(1,5,7,2)$. For bigger values of $d$, more interesting solutions can be found. 
   
\section{Polynomial bases}\label{appendix_polynomials}
From our point of view, the various generalizations of the Dehn--Sommerville relations are not just several linear relations, but the consequence of a single polynomial relation; see~\eqref{eq_f_fint_one} for the $f$-version and~\eqref{eq_polynomialDehnSommerville_hversion} for the $h$-version. For the $h$-version, ~\eqref{eq_polynomialDehnSommerville_hversion} is an equality between two polynomials which are written in different bases. 
The Dehn--Sommerville relations follow by transforming one basis in terms of the other and comparing the coefficients. 
In this appendix, we recall the bases we use.   

Let $\Z[x]$ be the ring of polynomials with integer coefficients in one variable $x$. 
Consider the following sets:
\begin{align}
M &\defs \left\{ x^i : 0\leq i \leq d \right\}, \\
\delta &\defs \left\{ (x+1)^ix^{d-i} : 0\leq i\leq d \right\}. 
\end{align}

The collection $M$ is the typical monomial basis. The collection $\delta$ is the second basis which is essential in this paper.  

\begin{proposition}\label{prop_polynomialBases}
The collections $M$ and $\delta$ are bases for the vector space of polynomials of degree less than or equal to $d$ in $\Z[x]$.
For $0\leq k\leq d$, the transformation from the monomial basis $M$ to the $\delta$-basis is given by:
\begin{align}\label{eq_changeBasis}
x^k = \sum_{i=0}^d (-1)^{d-k-i} {d-k \choose i} (x+1)^ix^{d-i}.
\end{align}
\end{proposition}

\begin{proof}
The collection $M$ is clearly a basis. Furthermore, $M$ and $\delta$ have the same cardinality. So, in order to show that $\delta$ is a basis it suffices to show that every element in $M$ can be written as a linear combination of elements in $\delta$. In other words, it suffices to show~\eqref{eq_changeBasis}. 
This relation follows from expanding the following expression using the binomial theorem:
\begin{align*}
x^k &= x^k \left( (x+1) - x \right)^{d-k} \\
&= x^k \sum_{i=0}^{d-k} {d-k \choose i} (x+1)^i(-x)^{d-k-i} \\
&= \sum_{i=0}^{d-k} (-1)^{d-k-i}{d-k \choose i} (x+1)^i x^{d-i}.
\end{align*}
This is equal to~\eqref{eq_changeBasis}, noting that ${d-k \choose i} = 0$ whenever $i>d-k$.  
\end{proof}

The generalization of the Dehn--Sommerville relations for balanced simplicial complexes also follows from a single polynomial relation~\eqref{eq_polynomialDehnSommerville_balanced_hversion}. 
In this case, two multivariate polynomials in $\Z[x_1,\dots , x_m]$ (the polynomial ring in $m$ variables) are equal to each other, and two analogous bases to $M$ and $\delta$ are used.

Let $\N\defs \{0,1,2,\dots\}$ be the set of non-negative integers, 
fix a sequence~$\bolda=(a_1,\dots,a_m)\in \mathbb N^m$ and let $\boldb=(b_1,\dots,b_m)\in \N^m$ with $\boldb\leq \bolda$. We denote $|b|\defs b_1+\dots +b_m$.
For simplicity, given a polynomial $p(x)\in \Z[x]$ in one variable $x$, we denote by $\x=(x_1,\dots ,x_m)$ a tuple of $m$ variables and define  the multivariate polynomial $p(\x)^\boldb\in \Z[x_1,\dots ,x_m]$  by
\begin{align}
p(\x)^\boldb \defs \prod_{i=1}^m p(x_i)^{b_i}.
\end{align}
With this notation, the binomial theorem becomes:
\begin{align}
\left( p(\x)+q(\x) \right)^\boldb = \sum_{\boldb'\leq \boldb} {\boldb \choose \boldb'} p(\x)^{\boldb'} q(\x)^{\boldb-\boldb'},  \label{eq_binomial_multi}
\end{align}
where ${\boldb \choose \boldb'} \defs \prod_{i=1}^m {b_i \choose b_i'}$. In particular, ${\boldb \choose \boldb'} = {\boldb \choose \boldb-\boldb'}$.

Now consider the following collections of polynomials of degree less than or equal to $\bolda$ in $\Z[x_1,\dots,x_m]$:

\begin{align}
\mathbf M &\defs \left\{ \x^\boldb : \boldb \leq \bolda \right\}, \\
\bm{\delta} &\defs \left\{ \x^\boldb(\x+\mathbf{1})^{\bolda-\boldb} : \boldb \leq \bolda \right\}. 
\end{align}

\begin{proposition}\label{prop_polynomialBases_multi}
The collections $\mathbf M$ and $\bm \delta$ are bases for the vector space of polynomials of degree less than or equal to $\bolda$ in $\Z[x_1,\dots ,x_m]$. For $\boldb \leq \bolda$, the transformation from the monomial basis $\mathbf M$ to the $\bm \delta$-basis is given by:
\begin{align}\label{eq_changeBasis_multi}
\x^\boldb = \sum_{\boldb \leq \boldb' \leq \bolda} (-1)^{|\boldb'|-|\boldb|} {\bolda - \boldb \choose \bolda-\boldb'}
 \x^{\boldb'}(\x+\mathbf{1})^{\bolda-\boldb'}.
\end{align}
\end{proposition}
\begin{proof}
Every polynomial  of degree less than or equal to $\bolda$ in $\Z[x_1,\dots ,x_m]$ can be written as a linear combination of elements in $\mathbf M$, and these elements are linearly independent. Thus, $\mathbf M$ is a basis.
Furthermore, $\mathbf M$ and $\bm \delta$ have the same cardinality. So, in order to show that $\bm \delta$ is a basis it suffices to show that every element in $\mathbf M$ can be written as a linear combination of elements in $\bm \delta$. In other words, it suffices to show~\eqref{eq_changeBasis_multi}. 

From~\eqref{eq_binomial_multi} we get
\begin{align*}
\x^\boldb &= \x^\boldb \left( (\x+\mathbf{1}) - \x\right)^{\bolda - \boldb} \\
& = \x^\boldb \sum_{\boldb'' \leq \bolda - \boldb} {\bolda-\boldb \choose \boldb''} (-\x)^{\boldb''} (\x+\mathbf{1})^{(\bolda - \boldb) - \boldb''}. 
\end{align*}
Taking $\boldb'=\boldb + \boldb''$ and using ${\bolda-\boldb \choose \boldb''} = {\bolda -\boldb \choose \bolda - \boldb'}$ we get
\begin{align*}
\x^\boldb = \sum_{\boldb \leq \boldb' \leq \bolda} (-1)^{|\boldb'|-|\boldb|} {\bolda -\boldb \choose \bolda - \boldb'} \x^{\boldb'}(\x+\mathbf{1})^{\bolda - \boldb'}.
\end{align*}
\end{proof}

\section{Connection to Stanley--Reisner rings}
\label{app_StanleyReisner}
Our proofs of the various generalizations of the Dehn--Sommerville relations are based in the $f=h$-reciprocity result in Theorem~\ref{thm_fh_reciprocity} and its multi-variate generalization in Theorem~\ref{thm_fh_reciprocity_multi}. These two theorems are the essence of the Dehn--Sommerville relations. In this appendix, we connect these results to a combinatorial reciprocity result of Stanley~\cite[Theorem~7.1]{stanley_combinatorics_algebra_book_1996} about the Hilbert series of Stanley--Reisner rings.    

Let $\SC$ be a finite simplicial complex of dimension $d-1$ on the vertex set $[n]=\{1,\dots,n\}$.
Given any field $\fieldk$, the \defn{Stanley--Reisner ring} $\SRring$ of the complex $\SC$ is defined as 
\begin{align*}
\SRring \defs \fieldk[x_1,\dots,x_n]/I_\SC,
\end{align*}
where 
\begin{align*}
I_\SC \defs (x_{i_1}\cdots x_{i_r} : i_1<\dots <i_r,\, \{i_1,\dots,i_r\}\notin \SC)
\end{align*}
is the ideal generated by the non-faces of the complex.
For $\boldlambda \defs (\lambda_1,\dots,\lambda_n)$, we denote by $F(\SRring,\boldlambda)$ the multi-graded \defn{Hilbert series} of $\SRring$. By counting the monomials of $\SRring$ according to their support $F\in \SC$, Stanley~\cite[Chapter II.1]{stanley_combinatorics_algebra_book_1996} provided the following nice expression for this Hilbert series:
\begin{align}
\label{thm_stanley_A}
F(\SRring,\boldlambda) =
\sum_{F\in \SC} \prod_{i\in F} \frac{\lambda_i}{1-\lambda_i}.
\end{align}

 Stanley also showed the following reciprocity result.

\begin{theorem}[{\cite[Theorem~7.1]{stanley_combinatorics_algebra_book_1996}}]
\label{thm_stanley_B}
The evaluation at $\mathbf{1}/\boldlambda$ of the Hilbert series of $\SRring$  satisfies
\begin{align}
(-1)^d F(\SRring,\mathbf{1}/\boldlambda) =
\sum_{F\in \SC} m_F \prod_{i\in F} \frac{\lambda_i}{1-\lambda_i}.
\end{align}
where $m_F=(-1)^{d-1-|F|}\widetilde \chi\bigl(\link{F}{\SC}\bigr)$. 
\end{theorem}

We highlight that the multiplicity $m_F$ here is the same multiplicity used in Theorem~\ref{thm_fh_reciprocity} and Theorem~\ref{thm_fh_reciprocity_multi}. In fact, we will show now that these two theorems are equivalent to specializations of Theorem~\ref{thm_stanley_B}.   

\begin{proposition}
\label{prop_SRring_fh_reciprocity}
Theorem~\ref{thm_fh_reciprocity} is equivalent to Theorem~\ref{thm_stanley_B} when replacing $\lambda_i$ by $\lambda$ for all $i$.
\end{proposition}

\begin{proof}
Replacing $\lambda_i$ by $\lambda$ in~\eqref{thm_stanley_A} we get 
\begin{align*}
F(\SRring,\lambda) 
&=\sum_{F\in \SC}  \frac{\lambda^{|F|}}{(1-\lambda)^{|F|}}  \\
&=\sum_{i=0}^{d}  \frac{f_{i-1}\lambda^i}{(1-\lambda)^i} 
= \frac{\sum_{i=0}^{d} f_{i-1}\lambda^i(1-\lambda)^{d-i}}{(1-\lambda)^d} \\
&= \frac{\sum_{i=0}^{d} h_i\lambda^i}{(1-\lambda)^d}   \quad \quad \quad \quad (\text{by}~\eqref{relation_fh_vectors}) \\
&= \frac{\tilde h(\lambda)}{(1-\lambda)^d}. 
\end{align*}

Replacing $\lambda_i$ by $\lambda$ in Theorem~\ref{thm_stanley_B} we get 
\begin{align*}
(-1)^d F(\SRring,1/\lambda) =
\sum_{F\in \SC} m_F  \frac{\lambda^{|F|}}{(1-\lambda)^{|F|}}.
\end{align*}

Combining the previous two equations, with the evaluation $\lambda=\frac{x}{x+1}$ in the second, we get $\frac{\lambda}{1-\lambda}=x$ and
\begin{align*}
\sum_{F\in \SC} m_F x^{|F|} 
&= (-1)^d F(\SRring,(x+1)/x) \\
&= x^d \tilde h\left(\frac{x+1}{x}\right),
\end{align*}
which is the statement of Theorem~\ref{thm_fh_reciprocity}.
Vice versa, taking $x=\frac{\lambda}{1-\lambda}$ we recover the evaluation $\lambda_i=\lambda$ of Theorem~\ref{thm_stanley_B} from Theorem~\ref{thm_fh_reciprocity}.  
\end{proof}

In the following proposition, we consider a balanced simplicial complex~$\SC$ of type $\bolda=(a_1,\dots,a_m)$ with vertex coloring $\kappa:[n]\rightarrow [m]$.
We also use extra variables $\boldomega=(\omega_1,\dots,\omega_m)$ for convenience. 

\begin{proposition}
Theorem~\ref{thm_fh_reciprocity_multi} is equivalent to Theorem~\ref{thm_stanley_B} when replacing $\lambda_i$ by $\omega_{\kappa(i)}$ for all $i$.
\end{proposition}

\begin{proof}
The proof is essentially the same as the proof of Proposition~\ref{prop_SRring_fh_reciprocity}.  

Replacing $\lambda_i$ by $\omega_{\kappa(i)}$ in~\eqref{thm_stanley_A} we get
\begin{align*}
F(\SRring,\boldomega) 
&=\sum_{F\in \SC} \prod_{i\in F} \frac{\omega_{\kappa(i)}}{1-\omega_{\kappa(i)}} \\
&= \sum_{F\in \SC} \left( \frac{\boldomega}{\mathbf{1}-\boldomega} \right)^{\boldb(F)}
= \frac{\sum_{F\in \SC} \boldomega^{\boldb(F)}(\mathbf{1}-\boldomega)^{\bolda-\boldb(F)}}{(\mathbf{1}-\boldomega)^\bolda} \\
&= \frac{\tilde h(\boldomega)}{(\mathbf{1}-\boldomega)^\bolda}    \quad \quad \quad \quad (\text{by}~\eqref{relation_flag_fh_vectors})
\end{align*}

Replacing $\lambda_i$ by $\omega_{\kappa(i)}$ in Theorem~\ref{thm_stanley_B} we get
\begin{align*}
(-1)^d F(\SRring,1/\boldomega) =
\sum_{F\in \SC} m_F  \left(\frac{\boldomega}{\mathbf{1}-\boldomega} \right)^{\boldb(F)}.
\end{align*}

Combining the previous two equations, with the evaluation $\boldomega=\frac{\x}{\x+\mathbf{1}}$ in the second, we get $\frac{\boldomega}{\mathbf{1}-\boldomega}=\x$ and
\begin{align*}
\sum_{F\in \SC} m_F \x^{\boldb(F)} 
&= (-1)^d F(\SRring,(\x+\mathbf{1})/\x) \\
&= \x^\bolda \tilde h\left(\frac{\x+\mathbf{1}}{\x}\right),
\end{align*}
which is the statement of Theorem~\ref{thm_fh_reciprocity_multi}.
Viceversa, taking $\x=\frac{\boldomega}{\mathbf{1}-\boldomega}$ we recover the evaluation $\lambda_i=\omega_{\kappa(i)}$ of Theorem~\ref{thm_stanley_B} from Theorem~\ref{thm_fh_reciprocity_multi}.
\end{proof}

\section*{Acknowledgements}
Part of the results in this manuscript were initially included in a preliminary preprint available at~\cite{ceballos_muehle_FHtriangles_DehnS0mmerville_2021}, which was split in two papers afterwards. Our discovery of the Dehn--Sommerville relations for manifolds and reciprocal complexes was motivated from our study of $F$- and $H$-triangles of $\nu$-associahedra. When we first posted the results about the Dehn--Sommerville relations we were not aware that many of them were already known. We are extremely grateful to Isabella Novik for her careful reading, valuable comments, solutions, and historical guidance. We are very grateful to Raman Sanyal for his comments that helped us to build the connection  in Appendix~\ref{app_StanleyReisner}. This appendix explains the relation between Theorems~\ref{thm_fh_reciprocity} and~\ref{thm_fh_reciprocity_multi} and Stanley's reciprocity result Theorem~\ref{thm_stanley_B} about the Hilbert series of the Stanley--Reisner ring.
We are also very grateful to Joseph Doolittle for many useful comments and discussions. 

\bibliography{literature}
\end{document}